\documentclass[11pt]{article}

\usepackage{amsmath, amssymb, amsthm, enumerate}
\usepackage{graphicx, caption, subcaption}
\usepackage{xcolor}
\usepackage{cancel} 
\usepackage[pagebackref]{hyperref}

\usepackage{xspace}

\usepackage{tikz}
\usetikzlibrary{arrows.meta,calc,decorations.markings}

\usepackage[margin=1.5in, marginparwidth=1.3in]{geometry}

\usepackage{titlesec}
\theoremstyle{plain}
\newtheorem{thm}{Theorem}[section]

\newtheorem{lemma}[thm]{Lemma}

\newtheorem{conj}[thm]{Conjecture}

\theoremstyle{definition}

\newtheorem{rmkx}[thm]{Remark}

\newcounter{bencomments}

\title{The Birthday Paradox for non-backtracking walks on regular graphs }  
\author{Benjamin Dozier \thanks{Department of Mathematics, Cornell University, \href{mailto:benjamin.dozier@cornell.edu}{\nolinkurl{benjamin.dozier@cornell.edu}}. Research supported in part by NSF Grant DMS-2247244 and the Simons Foundation.
}}

\begin{document}
\maketitle

\begin{abstract}
 We show a birthday paradox for random non-backtracking walk on regular graphs of degree at least $3$: such a walk of length $k$ has high probability of self-intersecting when $k$ is significantly greater than $\sqrt{n}$, where $n$ is the number of vertices of the graph.  This resolves a conjecture of Noga Alon and Yuval Peres for the fixed degree case.   
\end{abstract}



\section{Introduction}
\label{sec:intro}

The classical ``birthday paradox'' in elementary probability states that if we have $n$ objects and draw $k$ of them at random with replacement, uniformly and independently, we are likely to draw at least one object more than once provided that $k$ is significantly greater than $\sqrt n$.  In this paper we prove a version of this in which the objects are not drawn uniformly at random, but rather are the vertices of a random non-backtracking walk on a (regular) graph.

We consider (undirected) graphs without multiple edges or self-loops.   
When we speak of an edge, it will always come with an orientation.  
A \emph{non-backtracking walk} ($\mathrm{NBW}$) on such a graph is a sequence of edges such that (i) the endpoint of each edge (other than the last) is equal to the startpoint of the next, and (ii) no edge is followed by its opposite.   Such a walk has a \emph{self-intersection} if it hits some vertex more than once; a walk without self-intersection is called \emph{simple}.  We prove:

\begin{thm} \label{thm:main}
    Fix $d\ge 3$.  For any $\epsilon>0$, there exists a $C=C(d,\epsilon)$ such that if $\Gamma$ is a $d$-regular graph on $n$ vertices, and $k > C \sqrt{n}$, then for any edge $\delta$, the probability that a length $k$ random non-backtracking walk starting with $\delta$ is simple is less than $\epsilon$.   
  \end{thm}

  \begin{rmkx}
    For $d=2$ the statement is false.  For instance, if $\Gamma$ is a cycle of length $n$, then all NBWs of length $k<n$ are simple.  
  \end{rmkx}

  \begin{rmkx}
    If we remove the non-backtracking assumption, a much stronger statement is true.  In this case, at each step the walk has a definite chance of repeating the previous edge and thus producing a self-intersection.  Hence the relevant probability is related to a geometric random variable, and thus $C$ can be chosen (depending on $d,\epsilon$) so that the result holds for any $k>C$.  
  \end{rmkx}

  \begin{rmkx}
    The result is sharp in the sense that if we replace $\sqrt{n}$ in the statement by some function $f(n)$ with $f(n) = o(\sqrt n)$, then it becomes false.  This can be seen by considering (uniformly) \emph{random} $d$-regular graphs.  For closed loops, this is shown in \cite[Theorem 1.1]{ds}; as mentioned there, 
     for walks the same result holds and the argument is considerably easier (using the strategy of ``build both the walk and the graph step-by-step'').  
   \end{rmkx}

   Noga Alon and Yuval Peres conjectured the following more general result; our theorem above is the fixed degree case. 

  \begin{conj}[Alon-Peres]
    For any $\epsilon>0$, there exists a $C=C(\epsilon)$ such that if $\Gamma$ is a graph on $n$ vertices with all degrees at least $3$, and $k > C \sqrt{n}$, then for any edge $\delta$, the probability that a length $k$ random non-backtracking walk starting with $\delta$ is simple is less than $\epsilon$.   
  \end{conj}

  One could also explore analogs of these statements for (non-backtracking) \emph{closed loops}, rather than walks.

     \paragraph{Sketch of proof.}

  It turns out to be useful to do pre-sampling, i.e. work with walks with additional information, namely data that gives for each edge, a new edge to follow after that one; we call this a plan.  The ratio of simple plans to all plans is equal to the analogous ratio for walks.  Given a plan whose walk is simple, we produce many slight deviations, many of which will have to self-intersect.  In fact, many of the deviating paths, if simple, must return rather quickly to the original path, since otherwise, they would be forced to be disjoint (here is where it is important to consider plans, rather than plain walks).  And then of these deviations that eventually return to the original path, an excursion reversal argument shows that at least (approximately) half must self-intersect.

  The above produces for each simple plan, many self-intersecting ones.  Since the excursions are relatively short, each self-intersecting plan can only arise like this in a few ways.  Putting these facts together gives many more self-intersecting plans than simple ones.      

  \paragraph{Prior work on NBW and self-intersection.}
  Non-backtracking random walk on regular graphs has been studied in \cite{ABLES,LS,LP} among other papers.  Yadin has studied simple walks on graphs of large girth \cite{yadin}.  The relative number of simple and self-intersecting closed loops on random regular graphs plays an important role in the study of their spectral gap; see \cite{BS, Friedman, Friedman2008}.   In \cite{ds}, the author and Sapir show a birthday paradox for closed loops on random regular graphs.

  The birthday paradox in the related context of Markov processes is relevant to the effectiveness of the Pollard rho algorithm for discrete logarithms.  Results in this context have been proved in \cite{mv06, rho}.
  
After completing a draft of this paper, the author learned of independent work \cite{ev} on related questions, posted to the arxiv a few weeks before.  They give an upper bound on self-intersection time in a more general setting, but need an additional $\log n$ factor in their bound.  They also prove a bound without the $\log$ factor, but under the assumption that the graphs are fixed degree expanders.

  \paragraph{AI use and autoformalization.}
  The case of degree $3$ was proved by GPT 5.6 Sol Ultra, in about 2 hours of reasoning time.  
  That proof was digested by the author, who generalized it to all degrees, requiring some additional work.  The paper was written entirely by the author, who developed the organization, terminology, notation, exposition, and way that corner cases are handled.  But the key ideas of using ``plans'' (the author's terminology) to get disjointness of deviations, as well as the excursion reversal argument, were found by the AI.  
  The paper was proofread by the AI, and the author made several minor corrections based on the resulting feedback.  The figure was produced by the AI based on a hand-drawn sketch of the author.  

  Using Codex, the main theorem and proof have been (auto)formalized in Lean 4; see ancillary files.  These will also be made available on the \verb+Palomar+ registry.

  \paragraph{Acknowledgements.}
I am grateful to Noga Alon for telling me about his conjecture with Peres, and for later pointing out to me the related work \cite{ev}.  
I thank Ariel Yadin and Jenya Sapir for interesting discussions about the problem.  
  
\section{Proof of Theorem \ref{thm:main}}
\label{sec:proof}

Let $N(k,\delta)$ denote the number of length $k$ NBWs on $\Gamma$ starting with $\delta$, and let $N_{simp}(k,\delta)$ denote the number of those that are simple.  

\paragraph{Plans and their associated walks.}
A key idea in the proof is to consider walks with some additional information about where to go, even at unvisited vertices.   
Given an edge $\delta$, denote by $\bar \delta$ the edge with the reverse orientation.
We define a \emph{startless plan} on $\Gamma$ to be 
a function $\phi$ from the set of edges to itself such that (i) for each edge $(u,v)$, $\phi(u,v)$ equals $(v,w)$ for some vertex $w\ne u$, and (ii) $\phi$ is injective.   The map $\phi$ should be thought of as determining where a walker goes next given the current edge. Condition (i) corresponds to non-backtracking, while (ii) means that a single edge also determines a full past trajectory.  One can alternatively think of a startless plan as specifying for each vertex $v$ a derangement $\pi$ (permutation without fixed points) on the $d$ incoming edges; the restriction of $\phi$ to the edges ending at $v$ is then given by composing $\pi$ with the orientation reversing involution on edges.   


A \emph{plan} $\rho$ is a pair $(\phi, \delta)$ where $\phi$ is a startless plan and $\delta$ is an edge, the \emph{start}.  This determines an \emph{associated NBW} $\gamma_k(\rho)$ (of any length $k$ we choose) given by starting with $\delta$ and then repeatedly applying $\phi$.  We also define $\gamma(\rho)$ to be the infinite sequence produced like this.  We define a \emph{$k$-simple plan} to be one for which the associated walk of length $k$ is simple.
Let $\Omega^\delta$ denote the set of plans with start $\delta$, and let $\Omega^\delta_{simp}(k)$ be the set of $k$-simple plans with start $\delta$.

A crucial fact is that the ratio of the number of $k$-simple plans to all plans is exactly the probability that we are interested in.  The reason is that both can be thought of as encoding a choice of length $k$ walk, together with the same total number of other choices associated with edges that are not involved in the walk (for walks that are self-intersecting, this encoding is less direct).  We state this as a lemma, and give a more computational proof: 

\begin{lemma} \label{lemma:plan-reduction}
  For $k < n$:
    \begin{align}
    \frac{|\Omega^\delta_{simp}(k)|}{|\Omega^\delta|} = \frac{N_{simp}(k,\delta)}{N(k,\delta)}. 
\end{align}
\end{lemma}

\begin{proof}
    Let $R_d$ denote the number of derangements (permutations without fixed points) on the set $\{1,\ldots,d\}$.  Fixing some $1<j\le d$, let $R_d'$ be the number of these derangements that map $j$ to $1$ (note by symmetry, the number doesn't depend on the choice of $j$).   Since there are $d-1$ choices of this $j$, we have $R_d = (d-1)R_d'$.  
    We compute: 
    \begin{align}
    \frac{|\Omega^\delta_{simp}(k)|}{|\Omega^\delta|} = \frac{N_{simp}(k,\delta) \left( R_d' \right)^{k-1} \left(R_d\right)^{n-(k-1)}}{\left(R_d\right)^{k-1} \left(R_d\right)^{n-(k-1)}}  = \frac{N_{simp}(k,\delta)}{(d-1)^{k-1}} \\ = \frac{N_{simp}(k,\delta)}{N(k,\delta)}. 
    \end{align}

\end{proof}

\paragraph{Deviating from the plan.}
From here, to establish Theorem \ref{thm:main}, for each $k$-simple plan, we will consider many deviations away from it, producing many new walks, many of which will turn out to self-intersect (actually, we will only prove this for the \emph{average} over all $k$-simple plans with a fixed start).  The reason that the extra information of a plan is helpful is that it will force certain deviations to be disjoint.  

Our notion of deviation concerns changing a plan at a single vertex.  When $d=3$, there are only two choices for the data of a plan at a particular vertex, and thus there is a unique way to change the plan there.  However, when $d>3$, there are more choices, but what $\phi$ does to edges not along the associated walk will not be of much concern.  Nevertheless, it is useful to pick specific values for these (largely irrelevant) choices, which motivates the following definition; this will come up again in \eqref{item:middle} of proof of Lemma \ref{lemma:packing-reversing}.  
Suppose we have chosen a vertex $v$, a derangement $\pi$ of the incoming edges at $v$, and a pair $(\delta, \beta)$, where $\delta$ and $\beta$ are incoming and outgoing edges at $v$, respectively, with $\delta \ne \bar \beta$.  We pick arbitrarily a \emph{distinguished derangement contingency}, which is a derangement $\pi'$ of the incoming edges of $v$ such that $\overline{\pi'(\delta)}=\beta.$ (This choice will actually only affect parts of the plan that are not relevant to the associated walk, which is why the word ``contingency'' is used.)  

Given a $k$-simple plan $\rho = (\phi,\delta)$, consider the edges $\delta=\delta_1, \delta_2, \ldots $ of the associated walk $\gamma(\rho)$.  We define a \emph{deviation} of $\rho$ to be a pair $(\rho',\rho)$, where the plan $\rho'=(\phi',\delta)$ is obtained from $\phi$ by changing the derangement at the endpoint of $\delta_i$ for some $i=1,\ldots, k-1$ to the distinguished derangement contingency corresponding to $(\delta_i, \beta)$, for some $\beta\ne \delta_{i+1}$. 
Note that there are $(k-1)(d-2)$ deviations of a $k$-simple plan.  

We say a deviation $(\rho',\rho)$ obtained from changing $\phi(\delta_i)$ is $L$-\emph{simple} if its associated length $i+L$ walk is simple.  We want to choose $L$ large enough that many deviations are not $L$-simple.  

\begin{lemma} \label{lemma:packing-reversing} 
    There exists $C'$ (depending on $d$) such that if $L>C'n/k$ and $\delta$ is an edge of $\Gamma$, then at least $1/4$ proportion of the deviations of $k$-simple plans starting at $\delta$ are not $L$-simple.  \footnote{In this lemma, it is important that $d\ge 3$; when $d=2$, there are no deviations, and so the proportion is not meaningful.}
\end{lemma}

Note that deviations of two different plans could result in the same plan, but these are regarded as distinct deviations, since a deviation was defined as a pair $(\rho', \rho)$, i.e. it remembers the plan it deviated from.  

\begin{proof}
    There are two types of deviations of a $k$-simple plan $\rho$, and we will handle them with two different methods, a packing argument and an ``excursion reversal'' argument: 
    \begin{enumerate}[(A)]
    \item  \label{item:non-return} \emph{Non-returning} deviations $\rho'$ of $\rho$, i.e. those for which the associated walk does \emph{not} return to a vertex of $\gamma_k(\rho)$ within $L$ steps of deviating.  We will bound these with a packing argument using disjointness.

      Consider non-returning deviations that are $L$-simple, which in particular implies that the $L$ steps after deviating are on distinct edges.  Note that a plan determines the walk once a single edge along it is known (here we use that the function $\phi$ in the definition of plan is assumed to be injective). It follows that any two non-returning deviations $\rho',\rho''$ of $\rho$ must have the pieces consisting of the last $L$ edges of their associated walks disjoint.  Note that there are $nd$ edges of $\Gamma$, so we can have at most $nd/L$ such disjoint pieces.  So for $L>C'n/k$, we get that the number of these deviations is at most $\frac{nd}{C'n/k} =(d/C') k$, and choosing $C'>8d$, we get that there are fewer than $k/8$ such deviations of $\rho$.  In total there are $(k-1)(d-2)$ deviations of $\rho$, so by taking $C'$ a bit larger, we find that the proportion of deviations of $\rho$ that are non-returning and $L$-simple is less than $1/8$.  And so, a fortiori, considering \emph{all} possible $k$-simple $\rho$ with start $\delta$, we see that the proportion of deviations that are non-returning and $L$-simple is again at most $1/8$.  


    \item \label{item:return} \emph{Returning} deviations $\rho'$, i.e. those for which $\gamma(\rho')$ returns to a vertex $w$ of $\gamma_k(\rho)$ within $L$ steps of deviating (and we take $w$ to be the first such vertex along $\gamma(\rho')$).    We will separate into three cases based on $w$; the first two are unusual corner cases that can be dealt with quickly. 

      \begin{enumerate}[(i)] 
      \item \label{item:start} Suppose that $w$ is the startpoint of the start edge $\delta$.   Such a deviation is clearly not $L$-simple, and so no argument is needed here.

      \item \label{item:end} Suppose that $w$ is the endpoint of $\delta_k$.   Note that for fixed plan $\rho$, such a $\rho'$ is determined by the edge through which it comes back to $w$.  Hence there are at most $d-1$ such deviations.  The total number of deviations of $\rho$ is $(k-1)(d-2)$, so this accounts for at most $2/(k-1)$ proportion of deviations.
        For $k> 17$, this proportion is less than $1/8$. For $k\le 17$, note that we can take $C'>17$ and so $L>C'n/k \ge n$, and thus there are no $L$-simple deviations, since there are only $n$ vertices in our graph; in this case the proportion of such deviations that are also $L$-simple is zero.

      \item \label{item:middle} Suppose $\rho' $ is \emph{middle-returning}, i.e. returning but not of type \eqref{item:start} or \eqref{item:end}.  

      To handle these we define an ``excursion reversing'' involution $R$ on the set of all middle-returning deviations of $k$-simple plans starting at $\delta$.  This map will have the property that for each $(\rho',\rho)$, it cannot be that both $(\rho',\rho)$, $R(\rho',\rho)$ are $L$-simple.  The idea is that the plan is reversed along the edges of the ``excursion''.  See Figure \ref{fig:excursion-reversal}.  

      \begin{figure}
        \centering

\begingroup

\definecolor{excursionblue}{HTML}{0270FF}
\definecolor{reversalred}{HTML}{ED3725}
\definecolor{inverseyellow}{HTML}{FFD60A}

\tikzset{
  line cap=round,
  line join=round,
  base edge/.style={draw=black!72,line width=0.58pt},
  direction/.style={
    postaction={decorate},
    decoration={markings,
      mark=at position 0.58 with
        {\arrow{Stealth[length=1.7mm,width=1.25mm]}}}
  },
  upper route/.style={draw=excursionblue,line width=1.08pt},
  lower route/.style={draw=reversalred,line width=1.08pt},
  base vertex/.style={
    circle,draw=black!78,fill=black,line width=0.48pt,
    inner sep=0pt,minimum size=2.45pt
  },
  upper vertex/.style={
    circle,draw=excursionblue,fill=excursionblue,line width=0.45pt,
    inner sep=0pt,minimum size=3.15pt
  },
  upper interior/.style={
    circle,draw=black,fill=black,line width=0.45pt,
    inner sep=0pt,minimum size=2.75pt
  },
  lower vertex/.style={
    circle,draw=reversalred,fill=reversalred,line width=0.45pt,
    inner sep=0pt,minimum size=3.15pt
  },
  inverse vertex/.style={
    circle,draw=inverseyellow,fill=inverseyellow,
    line width=0.45pt,inner sep=0pt,minimum size=3.7pt
  },
  dart label/.style={font=\small,inner sep=1pt},
  path label/.style={font=\small,inner sep=1.5pt}
}

\begin{tikzpicture}[x=1cm,y=1cm]

\coordinate (t0)   at (0,0);
\coordinate (t1)   at (1.05,0);
\coordinate (t2)   at (2.15,0);
\coordinate (tim)  at (3.18,0);
\coordinate (ti)   at (4.23,0);
\coordinate (tlm)  at (8.00,0);
\coordinate (tw)   at (9.05,0);
\coordinate (tn)   at (10.10,0);
\coordinate (tend) at (11.15,0);

\coordinate (te1) at (5.08,0.90);
\coordinate (te2) at (5.98,1.30);
\coordinate (te3) at (7.28,1.30);
\coordinate (te4) at (8.20,0.90);

\node[path label,text=excursionblue,anchor=east] at (-0.18,-0.12)
  {$\gamma(\rho')$};

\draw[base edge,direction] (t0) -- (t1);
\draw[base edge,direction] (t1) -- (t2);
\draw[base edge] (t2) -- (2.46,0);
\node[font=\small] at (2.67,0.02) {$\cdots$};
\draw[base edge] (2.90,0) -- (tim);
\draw[base edge,direction] (tim) -- (ti);
\draw[base edge] (ti) -- (5.30,0);
\node[font=\small] at (6.12,0.02) {$\cdots$};
\draw[base edge] (6.94,0) -- (tlm);
\draw[base edge,direction] (tlm) -- (tw);
\draw[base edge,direction] (tw) -- (tn);
\draw[base edge,direction] (tn) -- (tend);

\draw[upper route] ([yshift=-2.65pt]t0) -- ([yshift=-2.65pt]t2);
\draw[upper route,densely dotted] ([yshift=-2.65pt]t2) -- ([yshift=-2.65pt]tim);
\draw[upper route]
  ([yshift=-2.65pt]tim) -- ([xshift=-4pt,yshift=-2.65pt]ti)
  .. controls +(.06,0) and +(-.05,-.08) .. (ti);

\draw[upper route,direction]
  (ti) -- node[path label,text=excursionblue,sloped,above=2pt] {$e_1$} (te1);
\draw[upper route,direction] (te1) -- (te2);
\draw[upper route,densely dotted] (te2) -- (te3);
\draw[upper route,direction] (te3) -- (te4);
\draw[upper route,direction]
  (te4) -- node[path label,text=excursionblue,sloped,above=2pt] {$e_j$} (tw);

\foreach \p in {t0,t1,t2,tim,tlm,tn,tend}
  \node[base vertex] at (\p) {};
\foreach \p in {t0,t1,tim}
  \node[base vertex,fill=black] at (\p) {};
\node[upper vertex] at (ti) {};
\node[base vertex,fill=black] at (tw) {};
\foreach \p in {te1,te2,te3,te4}
  \node[upper interior] at (\p) {};

\node[dart label,above=2.5pt] at ($(t0)!0.5!(t1)$) {$\delta_1$};
\node[dart label,above=2.5pt] at ($(t1)!0.5!(t2)$) {$\delta_2$};
\node[dart label,above=2.5pt] at ($(tim)!0.5!(ti)$) {$\delta_i$};
\node[dart label,above=2.5pt] at ($(tlm)!0.5!(tw)$) {$\delta_\ell$};
\node[dart label,below=5pt] at (tw) {$w$};

\coordinate (b0)   at (0,-2.75);
\coordinate (b1)   at (1.05,-2.75);
\coordinate (b2)   at (2.15,-2.75);
\coordinate (bim)  at (3.18,-2.75);
\coordinate (bi)   at (4.23,-2.75);
\coordinate (blm)  at (8.00,-2.75);
\coordinate (bw)   at (9.05,-2.75);
\coordinate (bn)   at (10.10,-2.75);
\coordinate (bend) at (11.15,-2.75);

\coordinate (be1) at (5.08,-1.85);
\coordinate (be2) at (5.98,-1.45);
\coordinate (be3) at (7.28,-1.45);
\coordinate (be4) at (8.20,-1.85);

\node[path label,text=reversalred,anchor=east] at (-0.18,-2.87)
  {$\gamma(\widetilde\rho')$};

\draw[base edge,direction] (b0) -- (b1);
\draw[base edge,direction] (b1) -- (b2);
\draw[base edge] (b2) -- (2.46,-2.75);
\node[font=\small] at (2.67,-2.73) {$\cdots$};
\draw[base edge] (2.90,-2.75) -- (bim);
\draw[base edge,direction] (bim) -- (bi);
\draw[base edge] (bi) -- (5.30,-2.75);
\node[font=\small] at (6.12,-2.73) {$\cdots$};
\draw[base edge] (6.94,-2.75) -- (blm);
\draw[base edge,direction] (blm) -- (bw);
\draw[base edge,direction] (bw) -- (bn);
\draw[base edge,direction] (bn) -- (bend);

\draw[lower route] ([yshift=-2.65pt]b0) -- ([yshift=-2.65pt]b2);
\draw[lower route,densely dotted] ([yshift=-2.65pt]b2) -- ([yshift=-2.65pt]bim);
\draw[lower route] ([yshift=-2.65pt]bim) -- ([yshift=-2.65pt]bi);
\draw[lower route] ([yshift=-2.65pt]bi) -- (5.30,-2.84);
\draw[lower route,densely dotted] (5.30,-2.84) -- (6.94,-2.84);
\draw[lower route]
  (6.94,-2.84) -- ([xshift=-4pt,yshift=-2.65pt]bw)
  .. controls +(.06,0) and +(-.05,-.08) .. (bw);

\draw[lower route,direction]
  (bw) -- node[path label,text=reversalred,sloped,above=2pt] {$\overline e_j$} (be4);
\draw[lower route,direction] (be4) -- (be3);
\draw[lower route,densely dotted] (be3) -- (be2);
\draw[lower route,direction] (be2) -- (be1);
\draw[lower route,direction]
  (be1) -- node[path label,text=reversalred,sloped,above=2pt] {$\overline e_1$} (bi);

\foreach \p in {be1,be2,be3,be4}
  \node[inverse vertex] at (\p) {};
\foreach \p in {b0,b1,b2,bim,blm,bn,bend}
  \node[base vertex] at (\p) {};
\foreach \p in {b0,b1,bim}
  \node[base vertex,fill=black] at (\p) {};
\node[lower vertex] at (bw) {};
\node[base vertex,fill=black] at (bi) {};

\node[dart label,above=2.5pt] at ($(b0)!0.5!(b1)$) {$\delta_1$};
\node[dart label,above=2.5pt] at ($(b1)!0.5!(b2)$) {$\delta_2$};
\node[dart label,above=2.5pt] at ($(bim)!0.5!(bi)$) {$\delta_i$};
\node[dart label,above=2.5pt] at ($(blm)!0.5!(bw)$) {$\delta_\ell$};
\node[dart label,below=5pt] at (bw) {$w$};

\end{tikzpicture}
\endgroup
        \caption{The mechanism by which at least half of middle-returning walks are self-intersecting.  The top figure shows the walk associated to simple plan $\rho$ (black edges), and a deviation $\rho'$ at the $i$th index (in blue).  The bottom figure represents the excursion reversal $R(\rho',\rho)=(\tilde \rho', \tilde \rho)$.  The plan $\tilde \rho$ is obtained by ``reversing'' along the excursion, i.e. the plan is changed at each yellow vertex by inverting the derangement there.  Then deviation at $w$ produces $\tilde \rho'$, whose associated walk is drawn in red.  In this diagram $\ell>i$, and $\gamma_k(\rho')$ is simple, but then $\gamma_k(\tilde\rho')$ cannot be -- it self-intersects at the endpoint of $\delta_i$.  
        }
        \label{fig:excursion-reversal}
      \end{figure}
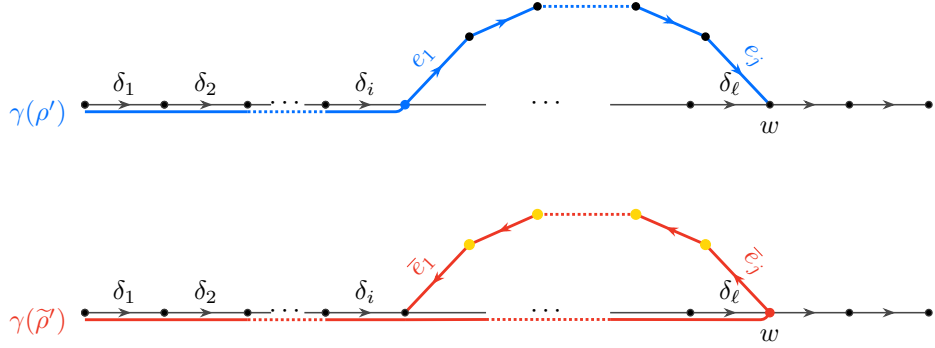

      Suppose that $\rho=(\phi,\delta)$, and that $\rho'=(\phi',\delta)$ deviates at index $i$.  The \emph{excursion} is the sequence of edges $e_1=\phi'(\delta_i), \ldots, e_j:=(\phi')^{\circ j} (\delta_i)$, where $e_j$ has endpoint $w$ on $\gamma_k(\rho)$, but no other $e_\eta$ with $\eta<j$ has endpoint on $\gamma_k(\rho)$.  By the returning assumption, $j\le L$.  We take $R(\rho',\rho)$ to be $(\tilde \rho', \tilde \rho)$, defined as follows.  For each vertex $v$ along the excursion (but not on $\gamma_k(\rho)$), we modify $\rho$ by taking the derangement $\pi$ at $v$ and replacing it with the inverse derangement.  The plan is not changed away from the excursion.  This gives us $\tilde \rho$.  To get $\tilde \rho'$, we modify $\tilde \rho $ only by changing the plan for $w$.  Since we are excluding deviations of type \eqref{item:end}, $w$ is an allowed location to deviate.  And since we are excluding type \eqref{item:start}, there is a (unique) edge $\delta_\ell$ of $\gamma_k(\rho)$ that has $w$ as its endpoint.   We define $\tilde \phi'$ at $w$ so that $\tilde \phi'(\delta_\ell)=\overline{e_j}$, using the distinguished derangement contingency that agrees with this.  

      These definitions mean that $\gamma(\rho')$ and $\gamma(\tilde \rho')$ both traverse the same excursion edges, but in opposite directions.
      
      If $\ell\le i$, i.e. the walk $\gamma(\rho')$ returns to $\gamma(\rho)$ at a vertex traversed no later than the excursion, then $\rho'$ is clearly not $L$-simple.  But otherwise $\tilde \rho'$ is not $L$-simple.  Thus $R$ has the properties claimed above.

    \end{enumerate}
    
      Combining the three cases above, we get that among all deviations of $k$-simple plans starting with $\delta$, at most $1/8+1/2=5/8$ are returning and $L$-simple.  
  
    \end{enumerate}

    Now any deviation is either non-returning or returning. Among all deviations of $k$-simple plans starting with $\delta$, non-returning $L$-simple constitute at most $1/8$ by \eqref{item:non-return}, while returning $L$-simple constitute at most $5/8$ by \eqref{item:return}.  Thus the total proportion of deviations that are $L$-simple is at most $5/8+1/8 = 3/4$, as desired.

\end{proof}

\begin{lemma} \label{lemma:plan-ratio}
  For any $\epsilon>0$, there exists such $C$ such that for $k^2 > Cn$ we have 
  \begin{align}
        \frac{|\Omega^\delta_{simp}(k)|}{|\Omega^\delta|}  < \epsilon.  \label{eq:ratio}
  \end{align}
\end{lemma}

\begin{proof}
  Choose $\eta>0$ small (see later in proof) and take $L$ an integer such that
  \begin{align}
    \label{eq:sandwich}
    \eta k > L > C'n/k,
  \end{align}
  where $C'$ is the constant in Lemma \ref{lemma:packing-reversing} ; there is enough gap between the left and right hand sides above to do this if we take $C$ a little larger than $C'/\eta$.  
  
  We will define a map $D$ that associates to each plan $\rho \in \Omega^\delta_{simp}(k)$ a subset of $\Omega^\delta$, which will consist of some of the plans formed from deviations of $\rho$.  Specifically, we consider deviations $(\rho',\rho)$, where $\rho'$ deviates at some step $1,\ldots, k-L$, and define $D(\rho)$ to be the set of such $\rho'$ that are not $L$-simple; the condition on the deviation step implies that $\gamma_k(\rho')$ is not simple.   We will use $D$ to bound the ratio in \eqref{eq:ratio}.

  First, we want to bound from below the average size of $D(\rho)$, over all $\rho \in \Omega^\delta_{simp}(k)$.  By Lemma \ref{lemma:packing-reversing} at least $1/4$ proportion of deviations are not $L$-simple, but that includes those deviating at any step, while for $D$ we imposed a restriction.  
   However, the proportion of deviations excluded by this restriction is at most $L/(k-1)$, and since $\eta k>L$, even with the imposed restriction we still get at least proportion $1/4 - 2\eta $ of deviations are allowed, and provided $\eta <1/16$, this is at least $1/8$. 
  Thus the average size of $D(\rho)$ is at least $(k-1)/8\ge k/16$.

  Second, we want to control the number of distinct plans that admit a common deviation, i.e. we want for each plan $\rho' \in \Omega^\delta$, a bound on the number of distinct plans $\rho$ for which $\rho' \in D(\rho)$.  To this end, note that if this number is at least one, then $\rho'$ must arise from a deviation of a $k$-simple plan, while $\gamma_k(\rho')$ is not simple.  Consider the smallest $j$ such that $\gamma_j(\rho')$ is not simple.  We claim that for any $\rho$ with $\rho'\in D(\rho)$ we must have that $\rho'$ deviates from $\rho$ at $i\in [j-L,j]$.  In fact, if $i>j$, then $\gamma_j(\rho),\gamma_j(\rho')$ would coincide, but the former is simple, while the latter is not.  On the other hand, if $i<j-L$, then, since $\rho'$ is not $L$-simple, its associated walk $\gamma(\rho')$ would have a self-intersection before step $j-L+L=j$, again a contradiction.  From this claim that we've established, it follows that the number of possible $\rho$ is at most $R_d\cdot (L+1)$, since $\rho, \rho'$ must agree except for the derangement at the single vertex (of which there are $R_d$ choices), and there are at most $j-(j-L)+1 = L+1$ choices of this vertex.

  Now combining the results about the map $D$ from the previous two paragraphs, we get that
  \begin{align*}
    |\Omega^\delta| &\ge \frac{ (k/16) \cdot |\Omega^\delta_{simp}(k)|}{ R_d \cdot (L+1)}\\
               & \ge \frac{(L/\eta)/16 \cdot |\Omega^\delta_{simp}(k)|}{ R_d \cdot (L+1)} \\
               & \ge \frac{ |\Omega^\delta_{simp}(k)|}{ 32\eta R_d},
  \end{align*}
  and so taking $\eta$ such that $32\eta R_d<\epsilon$ gives the desired result.  
\end{proof}

\begin{proof}[Proof of Theorem \ref{thm:main}]
  For $k\ge n$ the statement is trivial, since there are no simple walks of this length.  Now assuming $k<n$, we can apply Lemma \ref{lemma:plan-reduction} and then Lemma \ref{lemma:plan-ratio}
  (the $C$ in the theorem can be taken to be the square root of the $C$ in Lemma \ref{lemma:plan-ratio}).  
\end{proof}

{\footnotesize
\bibliographystyle{amsalpha}
  \bibliography{sources}
}

\end{document}